\documentclass{amsart}
\usepackage{amsthm}
\usepackage{amsmath}
\usepackage{amsfonts}
\usepackage[utf8]{inputenc}
\usepackage{makecell}
\usepackage{graphicx}
\usepackage{amsthm} 
\usepackage{amsmath}
\usepackage{amsfonts}
\usepackage{tikz}
\usepackage[capitalize, noabbrev]{cleveref}
\usepackage{tikz-cd}
\usepackage{amssymb}
\usepackage{filecontents}
\usetikzlibrary{3d,calc,intersections}
\usepackage{graphicx}
\usepackage{listings}
\usepackage{pgfplots}
\pgfplotsset{compat=1.15}
\usepackage{mathrsfs}
\usepackage{tikz-3dplot}
\usepackage{tikz,tikz-3dplot} 
\usetikzlibrary{arrows}
\definecolor{Frangipane}{RGB}{255, 217, 179}
\definecolor{Peach}{RGB}{255, 204, 152}

\usepackage{subcaption}

\newtheorem{theorem}{Theorem}[section]

\newtheorem{cor}[theorem]{Corollary}
\newtheorem{prop}[theorem]{Proposition}
\newtheorem{lemma}[theorem]{Lemma}
\theoremstyle{remark}{}
\newtheorem{rmk}[theorem]{Remark}
\theoremstyle{definition}

\usepackage[autostyle]{csquotes}
   \MakeAutoQuote{‘}{’}

\newtheorem{fact}[theorem]{Fact}

\newtheorem{convention}[theorem]{Convention}

\newcommand{\C}{\mathbb C}
\newcommand{\R}{\mathbb R}

\newcommand{\Mod }{\mathrm{Mod}}

\newcommand{\Z}{\mathbb Z}

\usepackage{graphicx} 

\title{Relative periods of abelian differentials in a prescribed stratum}
\date{}
\author{Thomas Le Fils}
\address{Sydney Mathematical Research Institute, The University of Sydney, NSW Australia}
\email{thomaslefils@gmail.com}
\begin{document}

\begin{abstract}
We  characterise the elements of $H^1(S, Z, \mathbb C)$, where $S$ is a closed surface and $Z\subset S$ is a finite set, that arise as the relative periods of an abelian differential in a given connected component of a stratum of their moduli space. 
This generalises a theorem obtained independently by Bainbridge, Johnson, Judge and Park and the author, and answers a question of Filip.
\end{abstract}
\maketitle

\section{Introduction}
The moduli space of abelian differentials $\Omega \mathcal M_g$ is the set of $(X, \omega)$ where $X$ is a closed Riemann surface of genus $g$ and $\omega$ is a non-zero holomorphic one-form on~$X$.
These spaces parametrise translation surfaces, which are fundamental objects of Teichm\"uller theory and low dimensional dynamics such as billiards in polygons, see for example \cite{MT02, Z06}.
The moduli space $\Omega\mathcal M_g$ is naturally stratified by the strata $\mathcal H(n_1, \ldots, n_k)$ formed by the abelian differentials with zeroes of multiplicity $n_1, \ldots, n_k$, where $\sum_i n_i = 2g - 2$.

Let us fix $S$ a closed oriented surface of genus $g$ and $Z\subset S$ a finite set of size $k$.
Given $(X, \omega)\in \mathcal H(n_1, \ldots, n_k)$ and an orientation preserving homeomorphism $f\colon S\to X$ that sends $Z$ to the zeroes of $\omega$, we define the period map of $(X, \omega)$ to be the element of $H^1(S, Z, \C) = \mathrm{Hom}(H_1(S, Z, \Z), \C)$ defined by $\gamma\mapsto \int_{f(\gamma)} \omega$.
The aim of this article is to determine the $\chi$ in $H^1(S, Z, \C)$ that arise as the period map of an abelian differential in a given connected component of a stratum, answering a question of Filip \cite[Question 3.1.15]{F24}.

Given $\chi\in H^1(S, Z, \C)$, we define an equivalence relation $\sim$ on $Z$ by $x\sim y$ if there exist $\alpha\in H_1(S, \mathbb Z)$ and $\beta$ a path from $x$ to $y$ such that $\chi(\alpha) = \chi(\beta)$.
The cardinalities of the equivalence classes form a partition of $k$ that we denote by $\Psi(\chi)$.
Poincar\'e duality gives a symplectic pairing on $H^1(S, \R)$. 
The evaluation of this pairing on real and imaginary parts of $\chi$ defines a map \[V\colon H^1(S, Z, \C)\to  H^1(S, \C) = H^1(S, \R) \oplus H^1(S, \R)\to \R.\]

\begin{theorem}\label{main_theorem}
Let $\chi\in H^1(S, Z, \C)$ and $\mathcal C\subset \mathcal H(n_1, \ldots, n_k)$ be a connected component of a stratum of abelian differentials.
There exists $(X, \omega)\in \mathcal C$ with period map $\chi$ if and only if the following two conditions hold:\newline
\begin{enumerate}
    \item  {\ } \hspace{4cm}$\displaystyle{V(\chi)>0}$.\newline
    \item If $\Lambda = \chi(H_1(S, \Z))$ is a lattice in $\C$, then there exists a partition  $A_1, \ldots, A_l$ of $\{1,\ldots, k\}$ such that the numbers $\mathrm{card}(A_i)$ form the partition $\Psi(\chi)$ and \[V(\chi)\geqslant \mathrm{Area}\left (\C/\Lambda\right )\max_{1\leqslant i\leqslant l}\sum_{j\in A_i} \left (n_j+1\right ).\]
\end{enumerate}
\end{theorem}

\begin{rmk}
Observe that the conditions of \cref{main_theorem} do not depend on the connected component $\mathcal C$. In particular the period map of an abelian differential in a given stratum can also be realised by other abelian differentials in any component of this stratum.
\end{rmk}

Each stratum $\mathcal H(n_1, \ldots, n_k)$ is locally modelled on $H^1(S, Z, \C)$.
Indeed, let $\widetilde {\mathcal H}(n_1, \ldots, n_k)$ be the cover of $\mathcal H(n_1, \ldots, n_k)$ consisting of $(X, \omega, f)$ where $(X, \omega)$ is in $ \mathcal H(n_1, \ldots, n_k)$ and $f\colon S\to X$ is an orientation preserving homeomorphism, sending $Z$ to the zeroes of $\omega$, up to isotopy.
The period maps define
\[\mathrm{Per}\colon \widetilde {\mathcal H}(n_1, \ldots, n_k)\to H^1(S, Z, \C)\]
that is well-known to be a local homeomorphism, see for example \cite[Proposition 6.3]{Yoccoz10}.
This gives the strata an affine structure that is fundamental in the study of the moduli space of abelian differentials. It is for example used to define the Masur-Veech volumes.
\cref{main_theorem} characterises the image of $\mathrm{Per}$.

\subsection*{Haupt's theorem}
The absolute period map of $(X, \omega)\in \Omega\mathcal M_g$, is the element of $H^1(S, \C) = \mathrm{Hom}(H_1(S, \Z), \C)$ defined by restricting its period map to absolute homology.
Haupt \cite{H20} characterised in 1920 the $\chi\in H^1(S, \C)$ that arise as period map of an abelian differential.
The datum of an abelian differential on $S$ is equivalent to that of a branched $(G, X)$ structure on $S$ in the sense of Thurston \cite[Chapter 3]{T97}, where $X = \R^2$ and $G = \R^2$ is the group of translations of $X$.
Haupt's theorem can be interpreted in this language as the characterisation of the holonomies of translation surfaces of genus $g$.
This result was rediscovered by Kapovich in \cite{K20} using a different approach, based on a dynamical study of the action of the mapping class group of $S$ on $H^1(S, \C)$.
\subsection*{Isoperiodic foliation}
Restricting the period maps to absolute homology defines
\[\mathrm{APer}\colon \widetilde{\Omega\mathcal M_g}\to H^1(S, \C)\]
where $\widetilde {\Omega\mathcal M_g} = \bigcup_{\sum_i n_i = 2g-2} \widetilde{\mathcal H}(n_1, \ldots, n_k)$.
The level sets of $\mathrm{APer}$ form leaves of a foliation of $\Omega\mathcal M_g$ called the isoperiodic foliation, or kernel foliation. This foliation has been extensively studied recently, see for example \cite{McMullenIso, H18, Ygouf,Winsor, CDF23}.
In \cite{CDF23}, Calsamiglia, Deroin and Francaviglia raised the question of characterising the absolute period maps of abelian differentials in a given stratum.
This amounts to determining the leaves of the isoperiodic foliation that intersect a given stratum of $\Omega\mathcal M_g$.
This question has been answered independently with different methods by Bainbridge, Johnson, Judge and Park \cite{BJJP22} and the author \cite{LF22}, refining the theorem of Haupt.

Note that \cref{main_theorem} easily implies the results of \cite{BJJP22, LF22} by taking all the $A_i$ to be singletons.
However, it is strictly stronger since the periods of paths joining different zeroes are not considered in \cite{BJJP22, LF22}.
The present article does not rely on them, and thus also provides a new proof of these results.
We follow the same strategy as \cite{LF22}: based on ideas and results of Kapovich, we study the mapping class group action of $S$ on the relative cohomology $H^1(S, Z, \C)$. 
The images of the strata by $\mathrm{Per}$ are open and invariant under this action.
This reduces the proof of \cref{main_theorem} to the construction of specific abelians differentials, that are covers of tori.

\subsection*{Other related results}
The results of \cite{BJJP22, LF22} have also been extended in other directions: in \cite{CFG22, CF24} to characterise periods of meromorphic differentials in a given stratum, and in \cite{LF21} to characterise the holonomies of branched complex projective structures in a prescribed stratum of their moduli space.

\subsection*{Tori covers}
We show in \cref{sec:MCG} that \cref{main_theorem} reduces to a problem of realisability of branched covers of the torus $T = \mathbb S^1\times \mathbb S^1$.
A branched cover $S\to T$ of degree $d$ induces branch data $\mathcal D = \{A_1, \ldots, A_k\}$, where $A_i = [n^i_1, \ldots, n^i_{l_i}]$ is a partition of $d$, given by the local degrees of the cover at each point of a fiber containing a branch point.
It follows from the Riemann-Hurwitz formula that $\sum_{i,j} \left (n^i_j  - 1\right ) $ is even.
We call a set $\mathcal D$ of partitions of $d$ satisfying this parity condition a candidate branch data.
The candidate branch data $\mathcal D$ is said to be realisable if it is the branch data of some cover $S\to T$.

\begin{theorem}\label{thm:tori_covers}
Let $\mathcal D = \{A_1, \ldots, A_k\} = \{ [n^1_1, \ldots, n^1_{l_1}], \ldots, [n^k_1, \ldots, n^k_{l_k}]\}$ be a candidate branch data.
Let $\mathcal C$ be a connected component of $\mathcal H\left (n^1_1-1, \ldots, n^k_{l_k} - 1\right )$.
There exists a cover $S\to T$ with branch data $\mathcal D$ that induces a surjection on first homology such that the induced flat structure on $S$ is in $\mathcal C$.
\end{theorem}
The problem of determining the candidate branch data $\mathcal D$ that are realisable, in the setting where $T$ is replaced with the sphere $\mathbb S^2$ is a subtle question that was raised by Hurwitz \cite{H01}.
We refer to \cite{P20} for a recent survey of progress regarding this problem, that is still open in full generality.
In the articles \cite{BSGDKZH03, BGK04}, the authors determine the candidate branch data that are realisable by a branched cover of the torus that is surjective at the level of fundamental groups. 
\cref{thm:tori_covers} refines this result, by prescribing the connected component of the stratum of abelian differentials the covering surface lies in.
Our proof does not rely on \cite{BSGDKZH03, BGK04} and thus provides a new proof of the existence of tori covers with prescribed branch data that are surjective on homology.

\subsection*{Organisation of the article}
In \cref{sec:reminder} we recall some known facts about abelian differentials and show that \cref{main_theorem} implies \cref{thm:tori_covers}.
We then discuss the action of mapping classes on cohomology in \cref{sec:MCG}, and prove that it suffices to construct specific examples of abelian differentials in each component to prove \cref{main_theorem}.
Then we provide constructions for each connected component: in \cref{sec:hyperelliptic} for the hyperelleptic components, in \cref{sec:odd} for the odd components, in \cref{sec:even} for the even components and for the remaining ones in \cref{sec:others}.
Finally, we provide a proof of \cref{main_theorem} for $g = 2$ in \cref{sec:genus_two}.
\subsection*{Acknowledgements}
The main result of this article was proven independently and simultaneously by Dawei Chen and Gianluca Faraco \cite{CF25} with different methods. I wish to thank them for our interesting discussions on our results and different approaches.
I thank Viveka Erlandsson, Scott Mullane, and Paul Norbury for their invitation and hospitality at the MATRIX program \emph{Teichmüller Theory and Flat Structures} in Creswick.
I thank the authors of the \emph{surface dynamics} package \cite{surface_dynamics} for {S}age{M}ath, which allowed me to easily experiment and simplify some constructions.
I also thank Simion Filip for his interest, and Maxime Wolff for his helpful suggestions about presentation.
This research was supported by an SMRI Postdoctoral Fellowship.

\section{Reminders on abelian differentials}\label{sec:reminder}
In this section we recall some well-known facts concerning abelian differentials and introduce notations.
We also explain the necessity of the two conditions of \cref{main_theorem}.
\subsection{Periods and volume}
One can concretely compute the function $V$ as follows: let $(a_1, b_1, \ldots, a_g, b_g)$ be a symplectic basis of $H_1(S, \Z)$.
Then for every  $\chi\colon H_1(S, Z, \Z)\to \R^2$,
\[V(\chi) = \sum_{i=1}^g \det \left (\chi\left (a_i\right ), \chi\left (b_i\right )\right ).\]
The relative periods of $(X, \omega)\in \Omega\mathcal M_g$ are the elements $\int_\gamma \omega$, where $\gamma\in H_1(X, Z, \Z)$ and $Z$ is the set of zeroes of $\omega$. The absolute periods of $(X, \omega)$ are the elements $\int_\gamma\omega$, where $\gamma\in H_1(S, \Z)$.
If $\chi$ is the period of $(X, \omega)$, then $V(\chi)$ is the volume of the flat metric defined by $(X, \omega)$, that is
\[V(\chi) = \frac{i}{2}\int \omega\wedge\overline{\omega}.\]
Indeed this follows from the Riemann bilinear relations, see for example \cite[Section 14]{Nara92}.
This explains the necessity of the first condition in \cref{main_theorem}.
Note that in particular the periods' span of $(X, \omega)$ is discrete if and only if it is a lattice in $\C$.

\subsection{Periods and covers of the torus}
Let us recall the well-known equivalence between branched tori covers and abelian differentials with discrete absolute periods.
Suppose that $p\colon S\to T$ is a branched cover of the torus $T = \C/\Lambda$  that induces a surjection on first homology.
One can pull back the abelian differential $dz$ of $T$ to get a complex structure $X$ on $S$ and an abelian differential $\omega$ on $X$.
The absolute periods of $\omega$ span $\Lambda$.
Conversely, if the absolute periods of $(X, \omega)$ span $\Lambda$ a lattice in $\C$, then let us fix $z_0\in X$. 
The map $z\mapsto\int_{z_0}^z \omega$ is a well-defined branched cover $p\colon X\to \C/\Lambda$ that is surjective at the first homology level.

We can recover the branch datum of $p$ from the relative periods of $(X, \omega)$.
Indeed, we have $\int_x^y \omega\in \Lambda$ if and only if $p(x) = p(y)$. 
Moreover, the degree $d$ of $p$ satisfies $V(\chi) = d\cdot \mathrm{Area}(\C / \Lambda)$.
That explains the necessity of the second condition of \cref{main_theorem}.
Recall that the datum of a cover $S\to T$ of degree $d$ branched over $x_1, \ldots, x_l$ amounts to that of a conjugacy class of representation \[\pi_1(T\setminus \{x_1, \ldots, x_l\})\to \mathfrak S_d,\] see for example \cite[Section 2]{BE79} for an interpretation of the work of Hurwitz \cite{H01}. One can thus move around the $x_i$ to change the periods of paths joining zeroes in different equivalence classes.
Let us denote by $\Psi(X, \omega)$ the partition $\Psi(\chi)$, where $\chi$ is the period map of $(X, \omega)$.

\begin{lemma}\label{lemma:wiggle_branched_points}
Let $(X_0, \omega_0)$ be an abelian differential with discrete absolute periods spanning $\Lambda$. Denote by $\mathcal C$ the connected component of its stratum.
Let $(x_i)_{i\leqslant l}\in T^l$, where $T = \C/\Lambda$, be the points over which $X_0\to T$ is branched.
For any $(z_i)_{i\leqslant l}$ pairwise distinct points in $T$, there exists $(X, \omega)\in \mathcal C$ with the same area as $(X_0, \omega_0)$, absolute periods spanning $\Lambda$ such that $X\to T$ is branched over $z_i$ and $\Psi(X, \omega) = \Psi(X_0,  \omega_0)$.
\end{lemma}
\begin{proof}
Let $R\to \mathrm{Conf_l}(T)$ be the fibration over the configuration space of the torus \[\mathrm{Conf_l}(T) = \{(x_1, \ldots, x_l)\in T\mid x_i\neq x_j \text{ if } i\neq j\}\] whose fiber over $B = (x_1, \ldots, x_l)$ is $\mathrm{Hom}(\pi_1(T\setminus B), \mathfrak{S}_d) / \mathfrak S_d$.
Let $\gamma\colon I\to \mathrm{Conf}_l(T)$ be a path from $(x_i)_i$ to $(z_i)_i$, where $I = [0, 1]$.
We can lift $\gamma$ to a path $\widetilde \gamma\colon I\to R$ such that $\widetilde \gamma(0)$ corresponds to $X_0\to T$.
The branched cover $\widetilde \gamma(1)$ defines an abelian differential $(X, \omega)$ satisfying the requirements.
\end{proof}
Observe that this dictionary between tori covers and abelian differentials with discrete absolute periods shows that \cref{thm:tori_covers} is equivalent to \cref{main_theorem} for representations with image of $H_1(S, \Z)$ discrete.
\section{Action of the mapping class group}\label{sec:MCG}
Let $\mathcal C$ be a connected component of the stratum $\mathcal H(n_1, \ldots, n_k)$, and $Z\subset S$ of size $k$.
Denote by $H(\mathcal C)$ the set of $\chi\in H^1(S, Z, \C)$ that arise as relative periods of abelian differentials in $\mathcal C$.
In other words, $H(\mathcal C)$ is the set of $\chi$ such that there exists a complex structure $X$ on $S$ and a holomorphic one-form $\omega$ on $X$ with zeroes at $Z$ such that $\chi(\gamma) = \int_{\gamma} \omega$ for every $\gamma\in H_1(S ,Z, \Z)$.
The periods are well-known to give a local homeomorphism between $\mathcal C$ and $H^1(S, Z, \C)$, see for example \cite[Proposition 6.3]{Yoccoz10}.
In particular $ H(\mathcal C)$ is open.
We often identify $H^1(S, Z, \C)$ with the set of representations $H_1(S, Z, \Z)\to \mathbb R^2$.
It carries an action of $\Mod(S, Z)$ by precomposition and an action of $\mathrm{GL}_2^+(\R)$ by postcomposition, where $\Mod(S, Z)$ is the group of isotopy classes of homeomorphisms of $S$ stabilising the set $Z$. 
\begin{lemma}
The set $H(\mathcal C)\subset H^1(S, Z, \C)$ is open and is invariant by both the action of $\Mod(S, Z)$ and $\mathrm{GL}_2^+(\R)$.
\end{lemma}
\subsection{Dynamics on the absolute parts}
We will use the work of Kapovich \cite{K20} and Calsamiglia-Deroin-Francaviglia \cite{CDF23} to derive properties of the closures of orbits of the mapping class group action on representations $\chi\colon H_1(S, Z, \Z)\to \C$.
The mapping class goup $\Mod(S)$ is the group of isotopy classes of homeomorphisms of $S$.
The natural forgetful map $\Mod(S, Z)\to \Mod (S)$ is surjective.
\begin{prop}\label{prop:MCG_generic}
Suppose that $g\geqslant 3$.
Let $\chi\in H^1(S, \C)$ be such that $V(\chi) > 0$, $\chi(H_1(S, \mathbb C))$ is not a lattice, and $N > 0$.
There exists $\chi_\infty$ in the closure of the $\mathrm{Mod}(S)$-orbit of $\chi$ such that both
\begin{enumerate}
    \item The image $\Lambda$ of $\chi$ is a lattice,
    \item The inequality $V(\chi) \geqslant N\cdot\mathrm{Area}(\C/\Lambda)$ holds.\label{C2}
\end{enumerate}
\end{prop}
\begin{proof}
Replacing $\chi$ with $A\cdot\chi$ for some $A\in\mathrm{GL}_2^+(\R)$, we can assume that the closure of $\mathrm{Mod}(S)\cdot\chi$ contains all the $\psi$ such that $V(\psi) = 1$ and $\Im \circ\psi$ has image $\Z$, see \cite[Proposition 3.10]{CDF23}.
We can define $\psi$ by, for example, $\psi(a_1) = 1 - \frac{1}{N}$, $\psi(a_2) = \frac{1}{N}$, $\psi(b_i) = i$ for $1\leqslant i\leqslant 2$, and $\psi(\gamma) = 0$ for $\gamma \in \{a_3, b_3, \ldots, a_g, b_g\}$.
\end{proof}
We can now prove \cref{main_theorem} for the $\chi\in H^1(S, Z, \C)$ such that $\Lambda = \chi( H_1(S, \Z))$ is not a lattice, assuming it holds for the ones where $\Lambda$ is a lattice and $g \geqslant 3$.
\begin{proof}
If $\chi(H_1(S, \Z)$ is not a lattice, then there exists $\chi_\infty$ in the closure of the $\Mod(S, Z)$-orbit of $\chi$ that is in $H(\mathcal C)$ by \cref{prop:MCG_generic} and \cref{main_theorem} in the lattice case. 
Since $H(\mathcal C)$ is open, the whole orbit of $\chi$ is in $H(\mathcal C)$.
\end{proof}
For the genus two case, we use an analysis of the author in \cite[Section 4]{LF22}.
\begin{lemma}\label{lemme:MCGgenus2}
Suppose that $g = 2$ and let $\chi\in H^1(S, \Z)$ be such $V(\chi) > 0$ and $\chi(H_1(S, \Z))$ is not a lattice.
There exist $A\in \mathrm{GL}_2^+(\R)$ and  $(a_1, b_1, a_2,b_2)$ a symplectic basis of $H_1(S, \Z)$ such that after replacing $\chi$ with $A\cdot\chi$,
\begin{enumerate}
    \item $\det\left (\chi(a_i), \chi(b_i)\right ) > 0$ for $1\leqslant i\leqslant 2$,
    \item The numbers $\Re(\chi(a_1))$ and $\Re(\chi(b_2))$ generate a dense group in $\R$.
\end{enumerate}
\end{lemma}
\begin{proof}
The first condition is proven in \cite[Section 4]{LF22}.
Let us prove the second one, assuming that $\det(\chi(a_i), \chi(b_i)) > 0$ for $i\in \{1,2\}$.
We can apply a matrix $A$ to assume that $\chi(a_1) = 1$ and $\chi(b_1) = i$.
Since $\chi(H_1(S, \Z))$ is not discrete, its projection on the real or imaginary line must be dense. 
If necessary, we can apply a rotation of angle $\frac \pi 4$ and then change $(a_1, b_1, a_2,b_2)$ to $(-b_1, a_1, a_2, b_2)$, in order to ensure that it is the case for the real line.
The group generated by $\Re(\chi(a_1))$,  $\Re(\chi(a_2))$ and $\Re(\chi(b_2))$ is dense, so one of $\Re(\chi(a_2))$ and $\Re(\chi(b_2))$ is not rational.
We may replace $(a_1, b_1, a_2,b_2)$ with $(a_1, b_1, b_2,-a_2)$ if necessary so that $\Re(\chi(a_1))$ and $\Re(\chi(b_2))$ generate a dense group.
\end{proof}

In \cite[Lemma 9.1]{CDF23}, Calsamiglia, Deroin and Francaviglia prove that the mapping class group acts transitively on the $\chi\in H^1(S, \mathbb C)$ such that $V(\chi) >0$ is fixed with image  a given lattice.
\begin{lemma}[Calsamiglia-Deroin-Francaviglia]\label{prop:MCG_covers}
The mapping class group $\Mod(S)$ acts transitively on the set of  $\chi\in H^1(S, \C)$ such that:
\begin{enumerate}
    \item the image of $\chi$ is a given lattice $\Lambda$,
    \item the ratio $V(\chi) / \mathrm{Vol}(\mathbb C/\Lambda)$ is a fixed integer $d$.
\end{enumerate}
\end{lemma}
\subsection{Dynamics on the relative parts}
Now that we have analysed the action of $\Mod(S)$ on the $\chi\in H^1(S, \C)$, let us focus on the action of $\Mod(S, Z)$ on the elements $\chi\in H^1(S, Z, \C)$.
Denote by $r\colon H^1(S, Z, \C)\to H^1(S, \C)$ the restriction map.
We observe that with point-pushing maps, see \cite[Chapter 4.2]{FarbMargalit},  one can change $\chi$ by adding any element of its absolute image $\chi\left (H_1\left (S, \Z\right )\right )$ to its relative parts.
\begin{lemma}[Point-Pushing]\label{lemma:point_pushing}
Let $\chi\in H^1(S, Z, \mathbb C)$ and $\psi\in \ker r$ be such that its image is included in $\chi\left (H_1(S, \Z)\right )$. The representation $\chi + \psi$ is in the $\mathrm{Mod}(S, Z)$-orbit of $\chi$.
\end{lemma}

\begin{proof}
Let $\chi$ and $\psi$ be such representations, and choose $z_0\in Z$. For $z\in Z\setminus\{z_0\}$, pick a path $\delta$ from $z_0$ to $z$.
There exists $\gamma\in \pi_1(S, z)$ such that $\psi(\delta) = \chi(\gamma)$.
We now act by the point-pushing map at $z$ following $\gamma$ to $\chi$.
This replaces $\chi(\delta)$ with $\chi(\delta) + \psi(\delta)$ and leaves unchanged $\chi(\alpha)$ for any path $\alpha$ joining $z_0$ to $z_1\neq z$.
We repeat this operation for each $z\neq z_0$, and replace $\chi$ with $\chi + \psi$.
\end{proof}
\subsection{Covers of the torus}
We can now present our main tool to prove \cref{main_theorem} in the case where $\chi(H_1(S, \Z))$ is a lattice.
\begin{prop}\label{prop:existence}
Let $\mathcal C$ be a connected component of a stratum $\mathcal H(n_1, \ldots, n_k)$ and $P$ be a partition of $k$. 
If there exists a translation surface $(X, \omega)$ in $\mathcal C$ of volume $d$ such that $\Psi(X, \omega) = P$ and its absolute periods span $\Z + i\Z$, then every $\chi\in H^1(S, Z, \C)$ is the period map of a surface in $\mathcal C$ if they satisfy the following:
\begin{enumerate}
    \item $\Psi(\chi) = P$,
    \item $\chi(H_1(S, \Z))$ is a lattice in $\C$ and its covolume $A$ satisfies $A\cdot d = V(\chi)$.
\end{enumerate}
\end{prop}
\begin{proof}
Let us fix such a $\chi\in H^1(S, Z, \C)$. Acting by an element of $\mathrm{GL}_2^+(\R)$, we can assume that $\chi(H_1(S, \Z))$ spans $\Lambda = \Z + i\Z$.
By \cref{lemma:wiggle_branched_points}, there exists an abelian differential $(X, \omega)\in \mathcal C$ of volume $d$ such that $\Psi(X, \omega) = P$ and with periods $\chi_0$  satisfying $\chi_0(H_1(S, \Z)) = \Lambda$, for all $\gamma\in H_1(S, Z, \C)$, there exists $\lambda\in \Lambda$ so that $\chi(\gamma) = \chi_0(\gamma) + \lambda$.
It follows from \cref{prop:MCG_covers} and \cref{lemma:point_pushing} that $\chi$ is in the $\Mod(S, Z)$-orbit of $\chi_0$. Hence $\chi$ belongs to  $H(\mathcal C)$.
\end{proof}

This analysis of the mapping class group dynamics reduces the proof of \cref{main_theorem} to the construction of a small number of translation surfaces.
The remainder of this article is devoted to the construction of such surfaces, in each connected component of the strata. This will give a case by case proof of \cref{main_theorem} with \cref{prop:existence}. 
We will also provide a different proof of \cref{main_theorem} for the genus two case.

\section{Hyperelliptic components}\label{sec:hyperelliptic}

In this section we prove \cref{main_theorem} for the hyperelliptic strata of the moduli spaces of abelian differentials. Let us recall from \cite{KZ03} that the spaces $\mathcal H(2g-2)$ and $\mathcal H(g-1, g-1)$ both have a connected component called hyperelliptic, such that all their abelian forms are double covers of a quadratic differential on the sphere.

\subsection{Single zero}\label{SingleZero}
Observe that the case of the stratra $\mathcal H^\mathrm{hyp}(2g-2)$ follows from the main result of \cite{BJJP22}, since in these minimal strata the absolute periods coincide with the relative ones. We still provide an explicit construction for completeness.
Let us recall a criterion to decide when an abelian differential of $\mathcal H(2g-2)$ belongs to the hyperelliptic component $\mathcal H^{\text{hyp}}(2g-2)$.
\begin{fact}
An abelian differential $(X,\omega)\in \mathcal H(2g-2)$ belongs to $\mathcal H^{\text{hyp}}(2g-2)$ if and only if there exists $\varphi\colon X\to X$ a holomorphic involution that preserves $\omega$ and has $2g+2$ fixed points.
\end{fact}
We now wish to construct a translation surface in $\mathcal H^\mathrm{hyp}(2g-2)$ whose periods span $\Lambda = \mathbb Z + i\mathbb Z$ of volume $d$ satisfying $d\geqslant 2g-1$.
We start by considering a rectangle with a horizontal side of length $d$ and a vertical side of length $1$. We glue its opposite sides to get a torus.
Now slit the torus along $2g-2$ horizontal segments of unit length, so that the right endpoint of the $i$-th coincides with the left endpoint of the slit number $i+1$, for $1\leqslant i < 2g-2$. 
Finally, we glue the bottom boundary of the slit number $i$ to the top boundary of the slit $2g-1-i$, see \cref{fig:hypersingle} for an example with $g = 3$.
Observe that the condition on $d$ ensures that we have enough space for our construction.
\begin{figure}[h]
\centering{
{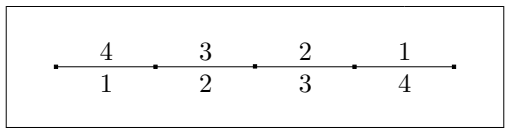}
\caption{Surface in $\mathcal H^{\text{hyp}}(4)$ with $d = 5$.}
\label{fig:hypersingle}
}
\end{figure}

The resulting surface $(X, \omega)$ has volume $d$. We easily check that it belongs to $\mathcal H(2g-2)$ and that its periods span $\mathbb Z + i\mathbb Z$. Moreover the involution $\varphi$ defined by a rotation of order two around the midpoint of union of the slit segments is such that $X/\varphi$ is homeomorphic to a sphere: it has $2g+2$ fixed points.

\begin{convention}
Most of our constructions will rely on same strategy: slitting and gluing back boundaries of a torus obtained by gluing opposite sides of a rectangle with horizontal side of length $d$ and unit length vertical side. Let us call such a torus a standard torus of length $d$.
\end{convention}

\subsection{Double zeroes}
We now turn to the component $\mathcal H^\mathrm{hyp}(g-1, g-1)$.
We want to exhibit a surface with lattice of absolute periods $\Lambda = \mathbb Z + i\mathbb Z$ and volume $d$ in $\mathcal H^\mathrm{hyp}(g-1, g-1)$ when the conditions of \cref{main_theorem} are satisfied.
The second condition of the statement of \cref{main_theorem} depends on whether the span of the absolute periods coincide with the span of the relative ones.
In other words, it depends on the number of points the associated cover is branched on.
\subsubsection{Relative periods different from absolute ones}\label{subsubsec:diff_hyp}
Let us first consider the case of where $P$ is of cardinality two: we want to construct a surface such that there exists a path joining its two singularities whose period is not in $\Lambda$.
The condition of \cref{main_theorem} is then $d\geqslant g$.
\begin{fact}
An abelian differential $(X,\omega)\in \mathcal H(g-1, g-1)$ belongs to $\mathcal H^{\text{hyp}}(g-1, g-1)$ if and only if there exists $\varphi\colon X\to X$ a holomorphic involution that preserves $\omega$, has $2g+2$ fixed points and interchanges the zeroes of $\omega$.
\end{fact}

Consider a standard torus of size $d$.
Let us slit it along $g-1$  horizontal line segments of length $\frac{1}{2}$ meeting at their endpoints to form a segment line of length $\frac{g-1}{2}$. We then slit along $g-1$ other horizontal line segments following the same pattern, on the right at distance $\frac{1}{2}$ of the previous one. 
Index the slits from the left to the right. We then glue the bottom boundary of the slit number $i$ with the top boundary of the slit $2g-1-i$. See \cref{fig:hyp22dif} for an example with $g = d = 4$.
\begin{figure}[h]
\centering{
{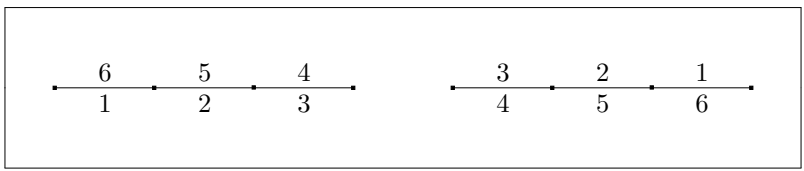}
\caption{Surface in $\mathcal H^{\text{hyp}}(3, 3)$ with $d = 4$.}
\label{fig:hyp22dif}
}
\end{figure}

The resulting surface $(X, \omega)$ is in $\mathcal H(g-1, g-1)$: the two zeroes of $\omega$ are at the endpoints of each slit. Moreover the absolute periods span $\mathbb Z + i\mathbb Z$, while the relative periods span $\frac{1}{2}\mathbb Z + i\mathbb Z$. Finally, its volume is $d$. Observe that the involution $\varphi$ defined by a rotation of order two centered at the middle of the two slitting segment lines has $2g+2$ fixed points and interchanges the zeroes of $\omega$. Therefore $(X, \omega)$ belongs to $\mathcal H^\mathrm{hyp}(g-1, g-1)$.

\subsubsection{Relative periods equal absolute ones}\label{subsubsec:eq_hyp}
We now turn to the case where the relative and absolute periods coincide.
The condition of \cref{main_theorem} in this case is that $d\geqslant 2g$.
We want to construct a torus cover branched over a single point, which is also called a square-tiled surface.
Let $r$ and $u$ be the permutations of $\mathfrak S_d$ defined as follows:
\begin{align*}
r &=  (1\ 2\ \ldots\ 2g-1),\\
u(i) &=  \left\{
\begin{array}{ll}
	2g-1-i & \mbox{if } 1\leqslant i \leqslant 2g-2\\
	i+1 & \mbox{if } 2g-1\leqslant i < d\\
    2g-1 & \mbox{if } i = d.
\end{array}
\right.
\end{align*}

Consider $d$ unit squares in the plane labelled with integers $1,\ldots, d$. We glue the right side of the square $i$ with the left side of the $r(i)$-th and its upper side with the lower side of the square number $u(i)$.
This way of defining square-tiled surfaces with pairs of permutations is classical, see for example \cite[Section 1]{M22}.
\cref{fig:hyp22eq} gives an example with $g = 4$ and $d = 8$. 
\begin{figure}[h]
\centering{
{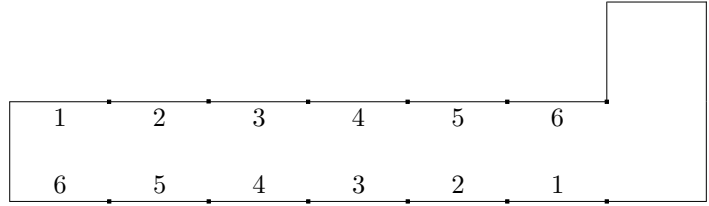}
\caption{Square-tiled surface in $\mathcal H^{\text{hyp}}(3, 3)$ with $d = 8$.}
\label{fig:hyp22eq}
}
\end{figure}

The resulting translation surface $(X, \omega)$ is of volume $d$, and both its absolute and relative periods span $\Z + i\Z$.
The commutator $\sigma = [r, u]$ is easily checked to be \[\sigma = \left (1~3~5~\ldots~2g-3~d\right )\left (2~5~2g-2~~2g-1\right).\]
It follows that $(X, \omega)\in \mathcal H(g-1, g-1)$, see \cite[Section 1.2]{M22}.
We may cut the squares $2g-1, \ldots, d$ in halves along vertical segment lines, and glue the resulting vertical cylinder back to the left of the horizontal cylinder formed by $1,\ldots, d$, as in \cref{fig:hyp22weier}.
This new presentation of $(X, \omega)$ allows us to define an involution $\varphi\colon X\to X$ preserving $\omega$ more easily: the rotation of order two along the centre of the horizontal cylinder.
We check that $\varphi$ has $2g+2$ fixed points, see \cref{fig:hyp22weier} for an example with $g = 4$.
\begin{figure}[h]
\centering{
{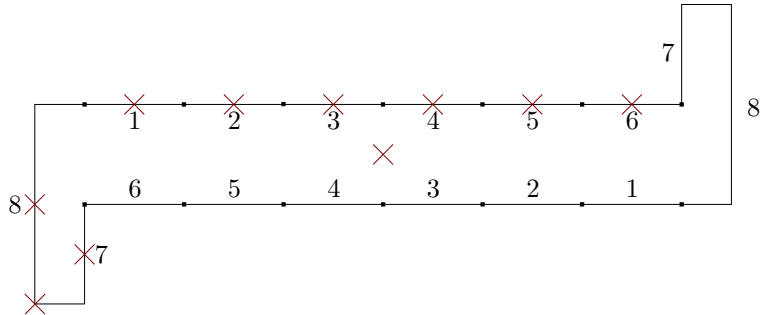}
\caption{Fixed points of the hyperelliptic involution.}
\label{fig:hyp22weier}
}
\end{figure}
Moreover, $\varphi$ interchanges the two zeroes of $\omega$, thus the abelian differential $(X, \omega)$ belongs to $\mathcal H^\mathrm{hyp}(g-1, g-1)$.
%
\subsection{Hyperelliptic cylinders}
Before constructing surfaces in the other components, let us provide a criterion to ensure that certain abelian differentials are not in the hyperelliptic components, following Zorich \cite[Proposition 5]{Z08}.
Suppose that $S$ has a single horizontal cylinder: the horizontal foliation on $S$ presents it as a cylinder $\mathcal C$ with its sides identified. Then the two boundaries of $\mathcal C$ decompose into saddle connections $T_1, \ldots, T_k$ and $B_1, \ldots, B_k$ that are naturally cyclically ordered.
The examples of single cylinders surfaces in the hyperelliptic components we provided are actually the only ones.
\begin{prop}[Zorich]\label{prop:single_cylinder_Zorich}
Let $S$ be a translation surface in a hyperelliptic component with a single cylinder $\mathcal C$ whose boundaries are decomposed as $T_1, \ldots, T_k$ and $B_1, \ldots, B_k$ as before. After cyclic permutation of the $T_i$, each saddle connection $B_i$ is glued to the saddle connection $T_{k - i + 1}$ in $S$.
\end{prop}
We sketch here the short proof of Zorich.
\begin{proof}
We may perturb $S = (X, \omega)$ so that the $T_i$ have pairwise different lengths.
Let $\varphi\colon S\to S$ be an involution that preserves $\omega$ such that $S/\varphi$ is a sphere.
We must have $\varphi^*(\omega) = -\omega$, hence $\varphi$ exchanges the $T_i$ and $B_j$, and reverses their cyclic order. Moreover each $T_i$ is sent to itself in $S$ since they have different lengths.
\end{proof}
As a consequence of \cref{prop:single_cylinder_Zorich}, we observe that it was necessary to construct a surface that does not have a single cylinder in \cref{subsubsec:eq_hyp}.
\begin{cor}
A square-tiled surface with $2g$ tiles in $\mathcal H^\mathrm{hyp}(g-1, g-1)$ whose absolute periods span $\Z + i\Z$ cannot have a single cylinder.
\end{cor}
\begin{proof}
Let us suppose that each boundary of the cylinder is decomposed into $k$ saddle connections.
It follows from \cref{prop:single_cylinder_Zorich} that the square-tiled surface $(X, \omega)$ is in $\mathcal H(k-1)$ if $k$ is odd and in $\mathcal H(\frac{k}{2}-1, \frac{k}{2}-1)$ if $k$ is even, see also \cref{fig:hypersingle} and \cref{fig:hyp22dif}.
Therefore we have $k = 2g$.
Hence $(X, \omega)$ is in the $\mathrm{SL}_2(\Z)$-orbit of the square-tiled surface defined by the permutations $r = \left ( 1~ 2~\ldots~2g\right )$ and $u = \left (1~2g\right )\left (2~2g-1\right)\ldots \left (g-1~g+1\right )$.
The two zeroes of $\omega$ are located respectively at the left and right corners of each square, and its absolute periods span $2\Z + i\Z$.
\end{proof}
\begin{rmk}
Note however that it is possible to construct a single cylinder square-tiled surface in  $\mathcal H^\mathrm{hyp}(g-1, g-1)$ with $d\geqslant2g+1$ tiles whose absolute periods span $\Z + i \Z$.
Indeed, one can then modify the construction of \cref{subsubsec:diff_hyp} with unit length slits and ensure that the two zeroes of the abelian differential are joined by a path of period an odd number.
The span of the absolute periods then coincides with the span of the relative ones.
\end{rmk}
This restriction on single cylinders surfaces in the hyperelliptic strata is reminiscent of the results of Jeffreys \cite{J21} who computed the minimal numbers of squares necessary to create square-tiled surfaces in each connected component with single horizontal and vertical cylinders.
\section{Odd components}\label{sec:odd}
\subsection{Arf invariant}
Let us suppose that $n_1\ldots, n_k$ are all even numbers such that $\sum_i n_i = 2g-2$.
We can associate to each connected component of $\mathcal H(n_1, \ldots, n_k)$ the parity of its spin structure. Let us recall a concrete way of defining it, following \cite[Section 9.4]{Z06}. Let $(X, \omega)\in \mathcal H(n_1, \ldots, n_k)$ and choose $(a_1, b_1, \ldots, a_g, b_g)$ a symplectic basis of $H_1(X, \Z)$.
For $\gamma\in H_1(X, \Z)$, pick a smooth $c:\mathbb S^1\to X$ representing $\gamma$ such that its derivative $c'$ never vanishes.
Define $\mathrm{Ind}(\gamma)$ to be the degree of the map $\mathbb S^1\to \mathbb S^1$ defined by $s\mapsto \frac{c'(s)}{|c'(s)|}$ in linear charts.
The quantity \[\mathrm{Arf}(X, \omega) = \sum_{i=1}^g \left (\mathrm{Ind}(a_i)+1\right )\left (\mathrm{Ind}(b_i)+1\right )\mod 2\]
does not depend on our choices, and is constant on each connected component.
This invariant, together with the hyperellipticity, completely determines the connected components of each stratum by the main result of \cite{KZ03}.
We now construct abelian differentials in the odd components of $\mathcal H(n_1, \ldots, n_k)$. We first present the construction and then show that it belongs to the right component.
\subsection{The minimal stratum}\label{subsec:minimal_odd}
Let us begin with the case $k = 1$, and suppose that $d\geqslant n_k+1$.
We slit a standard torus of size $d$ along unit length horizontal line segments following this pattern: we place two line segments having an extremity in common and then $n_k - 2$ others at their right, each separated from its previous slit by unit length.  The condition on $d$ ensures that our rectangle is long enough to slit along this pattern. We then glue the bottom of the slit number $i$, from left to right, with the top of the slit $i+1$ if $i < n$ and with the top of the first one if $i = n$.
See \cref{fig:oddsingle} for an example where $d= 7$ and $n_1 = 6$.
\begin{figure}[h]
\centering{
{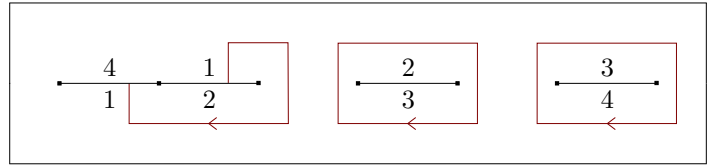}
\caption{Surface in $\mathcal H^\mathrm{odd}(6)$ with $d = 7$.}
\label{fig:oddsingle}
}
\end{figure}
\subsection{General construction}\label{subsec:general_construction}
Let us turn to the general construction for $k \geqslant 1$. 
Suppose that we have a partition $P = \{A_1, \ldots, A_l\}$ of $\{1, \ldots, k\}$ such that for every $1\leqslant i\leqslant l$, $\sum_{j\in A_i} \left (n_j + 1\right ) \leqslant d$.
We will slit a standard torus of length $d$ and glue back the boundaries with translations following several patterns.
We place a gluing pattern on $l$ different horizontal lines, one for each $A_i$.
Write $A_i = \{n_{j_1}, \ldots {n_{j_{k_i}}}\}$. 
On the line corresponding to $i$ we place side by side, from left to right, the patterns from \cref{subsec:minimal_odd} to create zeroes of order $n_{j_1}, n_{j_2},\ldots, n_{j_{k_i}}$.
Since $\sum_{j\in A_i} (n_j + 1)\leqslant d$, the rectangle is long enough to contain this concatenation.

\subsection{The resulting surface}\label{subsec:resulting_odd}
Let us denote by $(X, \omega)$ the abelian differential resulting from this construction.
It is of volume $d$, its absolute periods span $\Z + i\Z$ and $(X, \omega)$ belongs to $\mathcal H(n_1, \ldots, n_k)$.
Let us check that the Arf invariant of $(X, \omega)$ is odd.
Let $\gamma_1, \ldots \gamma_{g-1}\in H_1(S, \Z)$ be the simple closed curves going around each slit as the red curves in \cref{fig:oddsingle}.
There exist simple closed curves $\delta_1, \ldots, \delta_{g-1}$ in the interior of the rectangle such that $\delta_i$ intersects $\gamma_i$ once, and is disjoint from the $\gamma_j$ and the $\delta_j$, for $j\neq i$.
The sides $(a,b)$ of the standard torus, together with the elements $(\gamma_i, \delta_i)$, form a symplectic basis of $H_1(S, \Z)$. In this basis, one easily computes that $\mathrm{Arf}(X, \omega) = 1$ since $\mathrm{Ind}(\gamma_i) = 1$ for all $i$.
It remains to check that $(X, \omega)$ is not in the hyperelliptic component.
This can only potentially happen when $k \in \{1,2\}$.
One can reduce the length of the slits and align them to modify the surface inside its connected component and ensures that it has a single horizontal cylinder.
It thus follows from \cref{prop:single_cylinder_Zorich} that it does not belong to the hyperelliptic component, unless when $k=1$ and $g = 2$. 
But the hyperelliptic component of $\mathcal H(2)$ is the full stratum $\mathcal H(2)$ by \cite{KZ03}.

\section{Even component}\label{sec:even}
We now turn to the construction of surfaces in the even components of the strata $\mathcal H(n_1, \ldots, n_k)$, where each $n_i$ is even.
\subsection{The strata $\mathcal H(2, 2,\ldots, 2)$}
Let us first observe that the Arf invariants of the surfaces in $\mathcal H^{\mathrm{hyp}} (2, 2)$ we constructed in \cref{SingleZero} are even.
Indeed it is shown in \cite{KZ03} that the hyperelliptic component of $\mathcal H(2,2)$ coincides with the even one.
It can also be checked directly: it suffices to verify that there exists a surface in the hyperelliptic component of $\mathcal H(2,2)$ with even Arf invariant.
One can for example complete the loops of \cref{fig:H2222} to form a symplectic basis.
\begin{figure}[h]
\centering{
{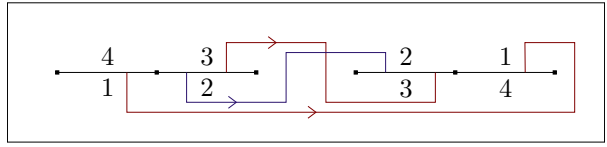}
\caption{Surface in $\mathcal H^\mathrm{hyp}(2, 2)$ with $d=6$.}
\label{fig:H2222}
}
\end{figure}
We also recall we need $d\geqslant 3$ to include this pattern in the standard torus.
Now suppose that we want to create an abelian differential with at least three zeroes, each of order two.
Let us assume that $P = \{A_1, \ldots, A_l\}$ is such that for every $1\leqslant i\leqslant l$, 
\[\sum_{j\in A_i}\left (n_j + 1\right )\leqslant d.\]
If every $A_i$ is a singleton, then we use the pattern of \cref{fig:H2222} on one horizontal line. 
On $l-2$ others we use the pattern from \cref{subsec:minimal_odd} to create a zero of order two.
If however one of the $A_i$, say $A_1$, is of cardinality at least two, then we place patterns on $l$ horizontal lines.
On the lines corresponding to $2\leqslant i\leqslant l$, we place side by side $\mathrm{Card}(A_i)$ patterns of \cref{subsec:minimal_odd} to create $\mathrm{card}(A_i)$ zeroes of order two. On the line corresponding to $A_i$ we place $\mathrm{Card}(A_1) - 2$ copies of the pattern of \cref{subsec:minimal_odd}, and the pattern of \cref{fig:H2222} with slits of unit length.
Observe that the length needed to include this line is \[3\left (\mathrm{Card(A_1) - 2}\right ) + 6 = \sum_{j\in A_1} \left (n_j + 1\right ).\]
We thus have constructed an abelian differential $(X, \omega)$ in $\mathcal H(2, 2, \ldots, 2)$ of volume $d$ with even Arf invariant, such that the partition $\Psi(X, \omega)$ corresponds to the one of $P$.
It remains to check that $(X, \omega)$ does not belong to the hyperelliptic component.
One can reduce the slits and align them to deform $(X, \omega)$ so that it has a single horizontal cylinder.
It then follows from \cref{prop:single_cylinder_Zorich} that $(X, \omega)$ is not in the hyperelliptic component.
\subsection{The minimal strata}\label{subsec:minimal_strata_even}
Let us now construct a surface in $\mathcal H^\mathrm{even}(2g-2)$ for $g\geqslant 4$.
Observe that the surface in $\mathcal H^\mathrm{hyp}(4)$ constructed in \cref{SingleZero} has even Arf invariant.
Indeed, the differentials $\mathcal H^\mathrm{hyp}(4)$ are the ones of $\mathcal H(4)$ with even spin invariant by \cite{KZ03}.
One could also consider a basis of its first homology group similar to \cref{fig:H2222} and check that its Arf invariant is even, see the four slits at the left in \cref{fig:even}. 
To construct a surface in $\mathcal H(2g-2)$ with even spin invariant, we use the following gluing pattern in the standard torus of length $d\geqslant 2g-1$: place the pattern of $\mathcal H(4)$ of \cref{fig:hypersingle}, and directly at its right endpoint, the one of $\mathcal H^\mathrm{odd}(2g-6)$ from \cref{subsec:minimal_odd}.
See \cref{fig:even} for an example with $g = 4$.  
\begin{figure}[h]
\centering{
{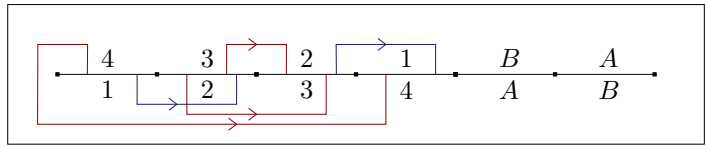}
\caption{Surface in $\mathcal H^\mathrm{even}(6)$ with $d = 7$.}
\label{fig:even}
}
\end{figure}
The resulting surface is in $\mathcal H(2g-2)$ and its periods span $\Z + i\Z$.
Moreover its Arf invariant is even, as one can check by completing the loops of \cref{fig:even} and \cref{fig:oddsingle} into a symplectic basis of the first homology group.
It follows from \cref{prop:single_cylinder_Zorich} that this surface does not belong to the hyperelliptic component, thus it is in $\mathcal H^\mathrm{even}(2g-2)$.

\subsection{General case}
We can now modify the construction of \cref{subsec:general_construction} to create surfaces in $\mathcal H(n_1, \ldots, n_k)$ with even Arf invariant, when $k\geqslant 2$ and one of the $n_i$ is at least four.
Indeed, it suffices to replace the part of the pattern creating one of the zero, say of order $n_1\geqslant 4$, by the pattern used in \cref{subsec:minimal_strata_even} to create a surface in $\mathcal H^\mathrm{even} (n_1)$.
The resulting surface is of volume $d$, its absolute periods span $\Z + i\Z$, and $\Psi(X, \omega)$ is given by the cardinalities of $A_1, \ldots A_l$.
It remains to check that it does not belong to the hyperelliptic component. 
This can only happen when $k = 2$.
But one can, as in \cref{subsec:resulting_odd}, reduce the slits and align them so that there is a single horizontal cylinder. 
It then follows from \cref{prop:single_cylinder_Zorich} that our surface does not belong to the hyperelliptic component.

\section{The other components}\label{sec:others}
In the previous sections, we have constructed tori covers in the hyperelliptic components of $\mathcal H(2g-2)$ and $\mathcal H(g-1, g-1)$. We also have formed surfaces in the odd and even components of $\mathcal H(n_1, \ldots, n_k)$ when all the $n_i$ are even.
In order to prove \cref{main_theorem} it remains to construct surfaces in the other connected components of strata. 
Namely, we need to form surfaces in the connected strata and in the non-hyperelliptic component of $\mathcal H(g-1, g-1)$, when $g$ is even.

\subsection{Constructing two zeroes of odd order}\label{subsec:two_odd_points}
Let us first present a gluing pattern in the standard torus of length $d$ to create two zeroes of order  $n \leqslant m$ that are odd, such that a path joining them has period in the span of the absolute periods.
Place two horizontal line segments of unit length, separated by distance one.
We slit along these segments, and glue the top boundaries with the bottom ones, creating two branch points of order one.
Directly at the left end of this pattern, we place the pattern of \cref{subsec:minimal_odd} for $n-1$ if $n > 1$. Directly at the right end of it, we place the pattern of \cref{subsec:minimal_odd} for $m - 1$ if $m > 1$. 
This changes the order of the two branch points to $n$ and $m$, see the first line of \cref{fig:HConnected} for an example with $n = 1$ and $m = 3$.
We have room to make these slits and gluing as long as $d\geqslant n + m + 2$.
This construction can also be described as creating four branch points, whose orders are $n-1, 1, 1, m-1$ and then shrinking two saddle connections as in \cite{Z08}.

\begin{figure}[h]
\centering{
{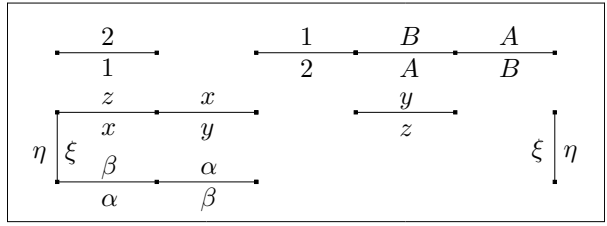}
\caption{Surface in $\mathcal H(1, 3, 3,5)$ with $d = 6$.}
\label{fig:HConnected}
}
\end{figure}

\subsection{General construction}
We now turn to the general case. Suppose that $d\geqslant 2$ and that $P = \{A_1, \ldots, A_l\}$ is a partition of $\{1, \ldots, k\}$ such that for all $1\leqslant i\leqslant l$ we have 
\[\sum_{j\in A_i} \left (n_j + 1\right ) \leqslant d.
\]
We are going to cut and glue the standard $d$ torus, following a pattern that has one line for each $A_i$, as in \cref{subsec:general_construction}.

\subsubsection{Even lines}\label{subsubsec:even_lines} If $A_i$ is such that $\sum_{j\in A_i} n_j$ is even, let us pair the $j$ in $A_i$ such that $n_j$ are odd. 
On a horizontal line, we place side by side different patterns presented before.
For each even $n_j$, we use the pattern from \cref{subsec:minimal_odd}. 
For each pair of odd $n_j$, we use the pattern explained in \cref{subsec:two_odd_points}.

\subsubsection{Odd lines}
We now pair the $1\leqslant i\leqslant l$ such that $\sum_{j\in A_i} n_j$ is odd. This is possible since $\sum_i n_i = 2g-2$.
In each member of such a pair, replace one of the $n_j$ by $n_{j}-1$. Now $\sum_{j\in A_i} n_j$ is even and we can use the pattern described in \cref{subsubsec:even_lines}.
However, one zero in each line needs its order to be increased by one.
Place the two lines next to each other, so that these two points are on the same vertical line.
We now slit along a vertical line segment joining them, and another vertical one that is obtained by an integer translation from the first, see \cref{fig:HConnected}.
Finally glue the sides of the vertical slits as indicated in \cref{fig:HConnected}.

\subsubsection{Exceptional lines}
In the special case where all the $j\in A_i$ satisfy $n_j = 1$ and $\mathrm{Card}(A_i)$ is even and at least four, then we do not use the pattern of \cref{subsubsec:even_lines}. Indeed, while it creates the required combinatorics of branch points, the span of the absolute period in this case is stricly included in $\Z + i\Z$.
In order to ensure that the absolute periods span $\Z + i\Z$, we replace two pairs of slits with the pattern of \cref{fig:H1111} to create four branch points of order two each.
\begin{figure}[h]
\centering{
{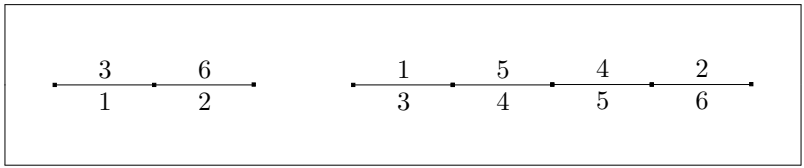}
\caption{Surface in $\mathcal H(1,1,1,1)$ with $d = 8$.}
\label{fig:H1111}
}
\end{figure}

\subsection{The resulting surface}
The abelian differential $(X, \omega)$ we constructed is of volume $d$ and is in $\mathcal H(n_1, \ldots, n_k)$.
Its absolute periods span $\Z + i\Z$ and $\Psi(X, \omega)$ is given by the partition $P$.
We just need to check that the resulting surface is not in the hyperelliptic component when $k = 2$ and $g$ is even.
There are two cases for our construction.
If the relative periods are the same as the absolute ones, then the pattern we used only has one line.
It has a single horizontal cylinder and \cref{prop:single_cylinder_Zorich} shows that it is not in the hyperelliptic component unless we are in the connected stratum $\mathcal H(1, 1)$.
If the set of relative periods differs from the absolute ones, that is when there are two lines, one can deform continuously the pattern created to align all the slits.
Then we are back to the single cylinder case and our surface is not in the hyperelliptic component unless it is in $\mathcal H(1,1)$ by \cref{prop:single_cylinder_Zorich}.
\section{Genus two}\label{sec:genus_two}
Recall that \cref{prop:MCG_generic} does not handle the genus two case.
In this section we thus suppose that $g = 2$.
To complete the proof of \cref{main_theorem}, it remains to show that the $\chi$ in $H^1(S, Z, \C)$ such that $\chi(H_1(S, \Z))$ is not a lattice and $V(\chi) > 0$  can be realised as periods of abelian differentials in each connected components of the strata of $\Omega\mathcal M_2$.
The two strata $\mathcal H(2)$ and $\mathcal H(1,1)$ of $\Omega\mathcal M_2$ are connected, by \cite{KZ03}.

By \cref{lemme:MCGgenus2}, there exists a symplectic basis $(a_1, b_1, a_2, b_2)$ such that $\Re(\chi(a_1))$ and $\Re(\chi(a_2))$ generate a dense group in $\R$ and $\det(\chi(a_i), \chi(b_i)) > 0$ for $i \in \{1,2\}$. We can further assume that $|\Re (\chi(b_2))| < |\Re(\chi(a_1))|$.
\subsection{The stratum $\mathcal H(2)$}
Let us say that a line segment in $\C$ is of holonomy $z\in \C$ if it is of the form $[z_0, z_0 + z]$.
Start with two parallelograms $P_i$ in $\C$ with sides of holonomy $\chi(a_i)$ and $\chi(b_i)$ for $i\in \{1,2\}$.
We glue the opposite sides of $P_1$ together with translations, as well as the opposite sides of holonomy $\chi(b_2)$ in $P_2$.
Let us consider a straight line segment $\delta$ in the torus $T_1$ formed by $P_1$ of holonomy $\chi(a_2)$. 
It is embedded in $T_2$ since $0 < |\Re(\chi(b_2))| < |\Re(\chi(a_1))|$.
We now slit $T_1$ along $\delta$. 
Glue the boundaries created in $T_1$ together with the boundaries of $P_2$ as in \cref{fig:H2}.
\begin{figure}[h]
\centering{
{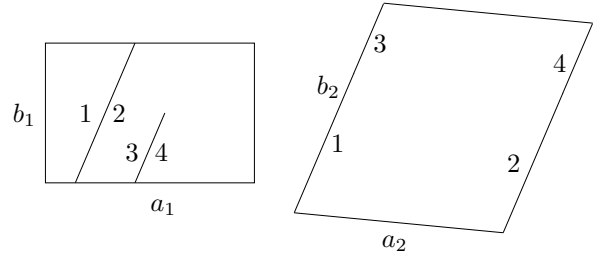}
\caption{Surface in $\mathcal H(2)$.}
\label{fig:H2}
}
\end{figure}
We can identify the resulting surface with $S$ so that its period map is given by $\chi$.
\subsection{The stratum $\mathcal H(1, 1)$}
We now turn to the construction of abelian differentials in $\mathcal H(1,1)$.
Let us construct an abelian differential with period map $\chi$.
Let $\delta$ be a simple path joining the two elements of $Z$.
By \cref{lemma:point_pushing}, we can act by $\Mod(S, Z)$ to ensure that 
\[0 < |\Re(\delta)| < \min \left (|\Re\left (\chi(a_1)\right )|, |\Re(\chi(a_2))|\right ).\]
Form two parallelograms $P_i$ with sides of holonomy $\chi(a_i)$ and $\chi(b_i)$ for $i\in \{1,2\}$.
We glue the opposite sides of each $P_i$ by translations to form two tori. 
Now slit each torus along a line segment of holonomy $\chi(\delta)$.
Our requirement on $\Re(\chi(\delta))$ ensures that these line segments embed in both tori.
It suffices now to glue back the opposite boundaries to form a surface in $\mathcal H(1,1)$ with the desired periods, see \cref{fig:H11}.
\begin{figure}[h!]
\centering{
{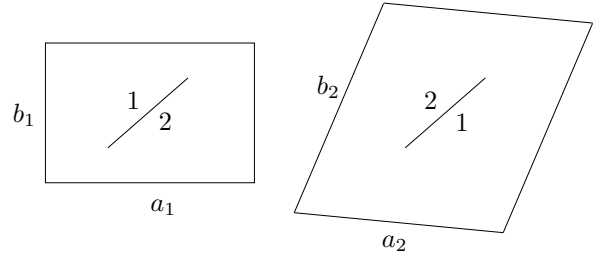}
\caption{Surface in $\mathcal H(1,1)$.}
\label{fig:H11}
}
\end{figure}

\bibliographystyle{alpha}
\bibliography{biblio.bib}

\newcommand{\etalchar}[1]{$^{#1}$}
\begin{thebibliography}{CDD{\etalchar{+}}25}

\bibitem[BE79]{BE79}
Israel Berstein and Allan~L. Edmonds.
\newblock On the construction of branched coverings of low-dimensional manifolds.
\newblock {\em Trans. Amer. Math. Soc.}, 247:87--124, 1979.

\bibitem[BGKZ03]{BSGDKZH03}
Semeon Bogatyi, Daciberg~L. Gon\c{c}alves, Elena Kudryavtseva, and Heiner Zieschang.
\newblock Realization of primitive branched coverings over closed surfaces following the {H}urwitz approach.
\newblock {\em Cent. Eur. J. Math.}, 1(2):184--197, 2003.

\bibitem[BGKZ04]{BGK04}
S.~A. Bogatyi, D.~L. Gon\c{c}alves, E.~A. Kudryavtseva, and H.~Zieschang.
\newblock Realization of primitive branched coverings over closed surfaces.
\newblock In {\em Advances in topological quantum field theory}, volume 179 of {\em NATO Sci. Ser. II Math. Phys. Chem.}, pages 297--316. Kluwer Acad. Publ., Dordrecht, 2004.

\bibitem[BJJP22]{BJJP22}
Matt Bainbridge, Chris Johnson, Chris Judge, and Insung Park.
\newblock Haupt's theorem for strata of abelian differentials.
\newblock {\em Israel J. Math.}, 252(1):429--459, 2022.

\bibitem[CDD{\etalchar{+}}25]{surface_dynamics}
Frédéric Chapoton, Diana Davis, Vincent Delecroix, Oscar Fontaine, Charles Fougeron, Luke Jeffreys, Samuel Lelièvre, Julian Rüth, Ivan Yakovlev, and Christopher Zhang.
\newblock Surface-dynamics, February 2025.

\bibitem[CDF23]{CDF23}
Gabriel Calsamiglia, Bertrand Deroin, and Stefano Francaviglia.
\newblock A transfer principle: from periods to isoperiodic foliations.
\newblock {\em Geom. Funct. Anal.}, 33(1):57--169, 2023.

\bibitem[CF24]{CF24}
Dawei Chen and Gianluca Faraco.
\newblock Period realization of meromorphic differentials with prescribed invariants.
\newblock {\em Forum Math. Sigma}, 12:Paper No. e90, 114, 2024.

\bibitem[CF25]{CF25}
Dawei Chen and Gianluca Faraco.
\newblock Relative period realization of holomorphic differentials with prescribed invariants.
\newblock {\em To appear}, 2025.

\bibitem[CFG22]{CFG22}
Shabarish Chenakkod, Gianluca Faraco, and Subhojoy Gupta.
\newblock Translation surfaces and periods of meromorphic differentials.
\newblock {\em Proc. Lond. Math. Soc. (3)}, 124(4):478--557, 2022.

\bibitem[Fil24]{F24}
Simion Filip.
\newblock Translation surfaces: dynamics and {H}odge theory.
\newblock {\em EMS Surv. Math. Sci.}, 11(1):63--151, 2024.

\bibitem[FM01]{FarbMargalit}
Benson Farb and Dan Margalit.
\newblock {\em A Primer on Mapping Class Groups}.
\newblock Princeton University Press, 2001.

\bibitem[Ham18]{H18}
Ursula Hamenst\"adt.
\newblock Ergodicity of the absolute period foliation.
\newblock {\em Israel J. Math.}, 225(2):661--680, 2018.

\bibitem[Hau20]{H20}
Otto Haupt.
\newblock Ein {S}atz \"uber die {A}belschen {I}ntegrale 1. {G}attung.
\newblock {\em Math. Z.}, 6(3-4):219--237, 1920.

\bibitem[Hur01]{H01}
A.~Hurwitz.
\newblock Ueber die {A}nzahl der {R}iemann'schen {F}l\"achen mit gegebenen {V}erzweigungspunkten.
\newblock {\em Math. Ann.}, 55(1):53--66, 1901.

\bibitem[Jef21]{J21}
Luke Jeffreys.
\newblock Single-cylinder square-tiled surfaces and the ubiquity of ratio-optimising pseudo-{A}nosovs.
\newblock {\em Trans. Amer. Math. Soc.}, 374(8):5739--5781, 2021.

\bibitem[Kap20]{K20}
Michael Kapovich.
\newblock Periods of abelian differentials and dynamics.
\newblock In {\em Dynamics: topology and numbers}, volume 744 of {\em Contemp. Math.}, pages 297--315. Amer. Math. Soc., Providence, RI, 2020.

\bibitem[KZ03]{KZ03}
Maxim Kontsevich and Anton Zorich.
\newblock Connected components of the moduli spaces of {A}belian differentials with prescribed singularities.
\newblock {\em Invent. Math.}, 153(3):631--678, 2003.

\bibitem[LF22]{LF22}
Thomas Le~Fils.
\newblock Periods of abelian differentials with prescribed singularities.
\newblock {\em Int. Math. Res. Not. IMRN}, (8):5601--5616, 2022.

\bibitem[LF23]{LF21}
Thomas Le~Fils.
\newblock Holonomy of complex projective structures on surfaces with prescribed branch data.
\newblock {\em J. Topol.}, 16(1):430--487, 2023.

\bibitem[Mat22]{M22}
Carlos Matheus.
\newblock Three lectures on square-tiled surfaces.
\newblock In {\em Teichm\"{u}ller theory and dynamics}, volume~58 of {\em Panor. Synth\`eses}, pages 77--99. Soc. Math. France, Paris, 2022.

\bibitem[McM14]{McMullenIso}
Curtis~T. McMullen.
\newblock Moduli spaces of isoperiodic forms on {R}iemann surfaces.
\newblock {\em Duke Math. J.}, 163(12):2271--2323, 2014.

\bibitem[MT02]{MT02}
Howard Masur and Serge Tabachnikov.
\newblock Rational billiards and flat structures.
\newblock In {\em Handbook of dynamical systems, {V}ol.\ 1{A}}, pages 1015--1089. North-Holland, Amsterdam, 2002.

\bibitem[Nar92]{Nara92}
Raghavan Narasimhan.
\newblock {\em Compact {R}iemann surfaces}.
\newblock Lectures in Mathematics ETH Z\"urich. Birkh\"auser Verlag, Basel, 1992.

\bibitem[Pet20]{P20}
Carlo Petronio.
\newblock The {H}urwitz existence problem for surface branched covers.
\newblock {\em Winter Braids Lect. Notes}, 7:Exp. No. 2, 43, 2020.

\bibitem[Thu97]{T97}
William~P. Thurston.
\newblock {\em Three-dimensional geometry and topology. {V}ol. 1}, volume~35 of {\em Princeton Mathematical Series}.
\newblock Princeton University Press, Princeton, NJ, 1997.

\bibitem[Win22]{Winsor}
Karl Winsor.
\newblock {\em Dynamics and {T}opology of {A}bsolute {P}eriod {F}oliations of {S}trata of {H}olomorphic 1-{F}orms}.
\newblock ProQuest LLC, Ann Arbor, MI, 2022.
\newblock Thesis (Ph.D.)--Harvard University.

\bibitem[Ygo21]{Ygouf}
Florent Ygouf.
\newblock {Isoperiodic Dynamics in Rank 1 Affine Invariant Orbifolds}.
\newblock {\em Int. Math. Res. Not. IMRN}, 09 2021.
\newblock rnab153.

\bibitem[Yoc10]{Yoccoz10}
Jean-Christophe Yoccoz.
\newblock Interval exchange maps and translation surfaces.
\newblock In {\em Homogeneous flows, moduli spaces and arithmetic}, volume~10 of {\em Clay Math. Proc.}, pages 1--69. Amer. Math. Soc., Providence, RI, 2010.

\bibitem[Zor06]{Z06}
Anton Zorich.
\newblock Flat surfaces.
\newblock In {\em Frontiers in number theory, physics, and geometry. {I}}, pages 437--583. Springer, Berlin, 2006.

\bibitem[Zor08]{Z08}
Anton Zorich.
\newblock Explicit {J}enkins-{S}trebel representatives of all strata of abelian and quadratic differentials.
\newblock {\em J. Mod. Dyn.}, 2(1):139--185, 2008.

\end{thebibliography}
\end{document}